\documentclass[11pt]{article}

\usepackage{amsmath}
\usepackage{amsfonts}
\usepackage{color}
\usepackage{amssymb}
\usepackage{graphicx}
\usepackage{hyperref}
\usepackage{enumerate}
\usepackage[all]{xy}

 \allowdisplaybreaks

\begin{document}

\newtheorem{problem}{Problem}

\newtheorem{theorem}{Theorem}[section]
\newtheorem{corollary}[theorem]{Corollary}
\newtheorem{definition}[theorem]{Definition}
\newtheorem{conjecture}[theorem]{Conjecture}
\newtheorem{question}[theorem]{Question}
\newtheorem{lemma}[theorem]{Lemma}
\newtheorem{proposition}[theorem]{Proposition}
\newtheorem{quest}[theorem]{Question}
\newtheorem{example}[theorem]{Example}

\newenvironment{proof}{\noindent {\bf
Proof.}}{\rule{2mm}{2mm}\par\medskip}

\newenvironment{proofof3}{\noindent {\bf
Proof of  Theorem 1.2.}}{\rule{2mm}{2mm}\par\medskip}

\newenvironment{proofof5}{\noindent {\bf
Proof of  Theorem 1.3.}}{\rule{2mm}{2mm}\par\medskip}

\newcommand{\remark}{\medskip\par\noindent {\bf Remark.~~}}
\newcommand{\pp}{{\it p.}}
\newcommand{\de}{\em}

\title{  {Inequalities regarding partial trace and 
partial determinant}\thanks{This research was supported by  NSFC (Nos. 11671402, 11871479),  
Hunan Provincial Natural Science Foundation (2016JJ2138, 2018JJ2479) and  Mathematics and Interdisciplinary Sciences Project of CSU.
 E-mail addresses: liyongtaosx@outlook.com(Y. Li), fenglh@163.com (L. Feng), huangzhengmath@163.com(Z. Huang), wjliu6210@126.com(W. Liu, corresponding author).} }

\author{Yongtao Li$^a$, Lihua Feng$^b$, Zheng Huang$^b$, Weijun Liu$^{\dag,b}$\\
{\small ${}^a$College of Mathematics and Econometrics, Hunan University} \\
{\small Changsha, Hunan, 410082, P.R. China } \\
{\small $^b$School of Mathematics and Statistics, Central South University} \\
{\small New Campus, Changsha, Hunan, 410083, P.R. China. } }

\maketitle

\vspace{-0.5cm}

\begin{abstract}
In this paper, we first present   simple proofs of Choi's results \cite{Choi17}, 
  then we give a short alternative  proof  
for Fiedler and Markham's inequality \cite{FM94}. 
We also obtain additional matrix inequalities related to partial determinants. 
 \end{abstract}

{2010 Mathematics Subject Classication.  15A45, 15A60, 47B65.}

\section{Introduction}
Throughout the paper, we use the following standard notation. 
The set of $n\times n$ complex matrices is denoted by $\mathbb{M}_n(\mathbb{C})$, 
and the identity matrix of order $k$ by  $I_k$, or $I$ for short. 
In this paper, 
we are interested in complex block matrices. Let $\mathbb{M}_n(\mathbb{M}_k)$ 
be the set of complex matrices partitioned into $n\times n$ blocks 
with each block being $k\times k$. 
The element of $\mathbb{M}_n(\mathbb{M}_k)$ is usually written as ${ H}=[H_{ij}]_{i,j=1}^n$, 
where $H_{ij}\in \mathbb{M}_k$ for all $i,j$. 
It is known that 
the matrices $[\mathrm{det} (H_{ij})]_{i,j=1}^n$ and $[\mathrm{tr} (H_{ij})]_{i,j=1}^n$ 
are positive semidefinite whenever $[H_{ij}]_{i,j=1}^n$ is positive semidefinite, 
e.g., \cite[p. 221 and p. 237]{Zhang}.

If ${ H}=[H_{ij}]_{i,j=1}^n \in \mathbb{M}_n(\mathbb{M}_k)$ is a  positive semidefinite matrix, 
the classical Fischer's inequality \cite[p. 506]{Horn} says that 
\begin{equation}\label{Fischer}
 \prod\limits_{i=1}^k \det H_{ii} \ge \det { H}. 
\end{equation}

In 1961, Thompson \cite{Thom61} proved the following elegant determinantal inequality (\ref{thom}), 
which is an extention of Fischer's result (\ref{Fischer}). 
The main weapon of Thompson's proof 
is an identity of Grassmann products, see \cite{Lin16} for a short proof.  

\begin{theorem}
Let ${ H}=[H_{ij}]_{i,j=1}^n \in \mathbb{M}_n(\mathbb{M}_k)$ be positive semidefinite. Then 
\begin{equation} \label{thom}
\det \Bigl( [\det H_{ij}]_{i,j=1}^n \Bigr) \ge \det { H}.
\end{equation}
\end{theorem}

Fiedler and Markham (1994) proved an analogous determinantal inequality for trace. 
In fact, Minghua Lin pointed out that in the proof of \cite[Corollary 1]{FM94}, 
Fiedler and Markham used the superadditivity of determinant functional, 
which can be improved by Fan-Ky's determinantal inequality \cite{FanKy}, 
i.e., the log-concavity of the determinant over the positive semidefinite matrices. 
 Here we state the stronger version (\ref{eqfm}),  see \cite{LinZhang} for more details.

\begin{theorem}\label{Lin}
Let ${ H}=[H_{ij}]_{i,j=1}^n \in \mathbb{M}_n(\mathbb{M}_k)$ be positive semidefinite. Then 
\begin{equation} \label{eqfm}
\left(\frac{\det \bigl( [\mathrm{tr} H_{ij}]_{i,j=1}^n \bigr)}{k^n}\right)^k \ge \det { H}.
\end{equation}
\end{theorem}

Now we introduce the definition of partial traces, 
which comes from quantum information theory. 
Given ${ H}=[H_{ij}]_{i,j=1}^n $ with $H_{ij}\in \mathbb{M}_k$, 
the first partial trace (map) $H \mapsto \mathrm{tr}_1H \in \mathbb{M}_k$ is defined as the  
adjoint map of the imbedding map $X \mapsto I_n\otimes X\in \mathbb{M}_n\otimes \mathbb{M}_k$. 
Here ``$\otimes$'' stands for the tensor product (or named the Kronecker product). 
Correspondingly, the second partial trace (map) \cite[p. 12]{Petz} $H\mapsto \mathrm{tr}_2H\in \mathbb{M}_n$ is 
defined as the adjoint map of the imbedding map $Y\mapsto Y\otimes I_k \in \mathbb{M}_n\otimes 
\mathbb{M}_k$. Therefore, we have

\[ \langle I_n\otimes X,H\rangle =\langle X, \mathrm{tr}_1H \rangle ,
\quad \forall X\in \mathbb{M}_k; \]
and 
\[ \langle Y\otimes I_k, H\rangle =\langle Y,\mathrm{tr}_2 H\rangle, 
\quad \forall Y\in \mathbb{M}_n. \]
The visualized forms of the partial traces 
are actually given in  \cite[Proposition 4.3.10]{Bhatia07} as
\[ \mathrm{tr}_1 { H}=\sum\limits_{i=1}^n H_{ii},\quad 
\mathrm{tr}_2{ H}=[\mathrm{tr}H_{ij}]_{i,j=1}^n. \] 

It is easy to see that $\mathrm{tr}_1 H$ and $\mathrm{tr}_2H$ are positive semidefinite 
whenever $H$ is positive semidefinite. With what has been just defined, 
inequality (\ref{eqfm}) can be written as 
\begin{equation} \label{dettr2}
\left( \frac{\det (\mathrm{tr}_2H)}{k^n}\right)^k \ge \det H.
\end{equation}

Recently, Choi  introduced the definition of ``partial determinant'' and 
derived some interesting properties in \cite{Choi17}. 
For a given block matrix ${ H}$, 
imitating the appearance of $\mathrm{tr}_2 { H}$, 
a natural definition of $\det_2 { H}$ is given as 
\[ \mathrm{det}_2 { H}=[\det H_{ij}]_{i,j=1}^n \in \mathbb{M}_n. \]
However, it does not seem easy to give the definition 
 of $\mathrm{det}_1{ H}$ analogous to $\mathrm{tr}_1 { H}$. 
The following ingenious mind originated from Choi. 
For ${ H}=[H_{ij}]_{i,j=1}^n \in \mathbb{M}_n(\mathbb{M}_k)$, 
where $H_{i,j}=\left[h_{l,m}^{i,j}\right]_{l,m=1}^k$, we define   
$\mathrm{det}_1{ H}\in \mathbb{M}_k$ by 
\[ \mathrm{det}_1 { H}=[\det G_{lm}]_{l,m=1}^k, \]
where $G_{lm}=\left[h_{l,m}^{i,j}\right]_{i,j=1}^n$. 
For convenience, we will denote $\widetilde{ H}$ to be  
\[ \widetilde{H}=\left[ \left[ h_{l,m}^{i,j}\right]_{i,j=1}^n\right]_{l,m=1}^k 
\in \mathbb{M}_k(\mathbb{M}_n). \]

Motivated by (\ref{dettr2}), Choi   \cite[Theorem 6]{Choi17} proved 
\begin{theorem} \label{choi}
Let  $H\in \mathbb{M}_n(\mathbb{M}_k)$ be positive semidefinite. Then 
\begin{equation}\label{trdet1}
\left( \frac{\mathrm{tr}(\mathrm{det}_1 H)}{k}\right)^k\ge \det H.
\end{equation}
\end{theorem}
We will present an alternative proof later.

The paper is organized as follows. 
In Section 2, we shall 
present two alternative simple proofs for Fiedler and Markham's inequality (\ref{eqfm}) 
 and Choi's inequality (\ref{trdet1}), and then the equivalent relations between
partial traces and partial determinants are drawn. 
In Section 3, we shall give two extensions of partial determinant, 
and some related inequalities are included.

\section{Alternative proofs for (\ref{eqfm}) and (\ref{trdet1})}

If $A=[a_{ij}]$ is of order $m\times n$ and $B$ is $s\times t$, the tensor product of $A,B$, 
denoted by $A\otimes B$, is an $ms\times nt$ matrix, 
partitioned into $m\times n$ block matrix with the $(i,j)$ block the $s\times t$ matrix $a_{ij}B$. 
Let $\otimes^r A=A\otimes \cdots \otimes A$ be the $r$-fold tensor power of $A$, 
and we denote by $\wedge^r A$ the $r$-th Grassmann power (\cite[pp. 16-19]{Bhatia}) of $A$, 
which  is the same as the $r$-th multiplicative compound matrix of $A$, 
and also is a restriction of $\otimes^r A$. 
There are some basic properties of the tensor product, 
we  briefly list some items below.

\begin{proposition}\label{prop}
Let $A, B, C$ be matrices of appropriate sizes. Then \\
1. $(A\otimes B)\otimes C =A\otimes (B\otimes C)$.\\
2. $(A\otimes B)(C\otimes D)=(AC)\otimes (BD)$. \\
3. $(A\otimes B)^T=A^T\otimes B^T$. \\
4. $(A\otimes B)^{-1}=A^{-1}\otimes B^{-1}$ if $A$ and $B$ are invertible. \\
Furthermore, if $A,B,C$ are positive semidefinite matrices, then \\
5. $A\otimes B$ is positive semidefinite. \\
6. If $A\ge B$, then $A\otimes C\ge B\otimes C$. \\
7. $\otimes^r (A+B)\ge \otimes^r A+\otimes^r B$ for all positive integer $r$.
\end{proposition}

\begin{lemma}\label{tr1tr2}
For $H\in \mathbb{M}_n(\mathbb{M}_k)$, we have 
$\mathrm{tr_1}\widetilde{H}=\mathrm{tr}_2 H$ and $\mathrm{det}_1\widetilde{H}=\mathrm{det}_2H$.
\end{lemma}

\begin{proof}
It is straightforward.
\end{proof}

\begin{lemma} \label{lem3}
For $A\in \mathbb{M}_n$ and $B\in \mathbb{M}_k$, there exists a permutation matrix $P(n,k)$ 
of order $nk$ depending only on $n,k$ such that 
$\widetilde{A\otimes B}=P(n,k)^T(A\otimes B) P(n,k)$. 
\end{lemma}

\begin{proof}
Let $A=[a_{ij}]_{i,j=1}^n$ and $B=[b_{ij}]_{i,j=1}^k$. Since 
\[ A\otimes B= \left[a_{ij}B\right]_{i,j=1}^n
=\left[ \left[a_{ij}b_{lm}\right]_{l,m=1}^k \right]_{i,j=1}^n. \]
Therefore 
\[ \widetilde{A\otimes B} =\left[ \left[a_{ij}b_{lm}\right]_{i,j=1}^n \right]_{l,m=1}^k 
=\left[b_{lm}A\right]_{l,m=1}^k =B\otimes A.\]
Note that $B\otimes A$ is permutationnally similar to $A\otimes B$, see \cite[p. 40]{Zhan}, 
then there exists a permutation matrix $P(n,k)$ depending on $n,k$ such that 
\[ \widetilde{A\otimes B}=P(n,k)^T(A\otimes B) P(n,k). \]
The result follows. 
\end{proof}

\begin{theorem}
For ${ H}=[H_{ij}]_{i,j=1}^n \in \mathbb{M}_n(\mathbb{M}_k)$, 
$\widetilde{H}$ is permutationally similar to ${H}$.  
\end{theorem}

\begin{proof}
Here we present a   short proof which is quite different from that in  \cite{Choi17}. 
We first observe a known fact, for any ${H}\in \mathbb{M}_n(\mathbb{M}_k)$, 
we may write ${H}=\sum_{i=1}^m A_i\otimes B_i$ 
for some $A_i\in \mathbb{M}_n,B_i\in \mathbb{M}_k$ and some positive integer $1\le m\le n^2$. 
By Lemma \ref{lem3}, there is a permutation matrix $P(n,k)$ such that 
\[ \widetilde{H}=\sum\limits_{i=1}^m \widetilde{A_i\otimes B_i} =
\sum\limits_{i=1}^m P(n,k)^T(A_i\otimes B_i)P(n,k)=P(n,k)^T HP(n,k), \]
as desired.
\end{proof}

\noindent
{\bf Remark}~
By applying Fischer's inequality (\ref{Fischer}) to $\widetilde{H}$, we get 
\begin{equation}\label{gll}
 \det H=\det \widetilde{H}\le \prod\limits_{l=1}^k \det G_{ll}.
\end{equation}
The  inequality (\ref{gll}) is proved by using Koteljanskii's inequality in \cite{Choi17}.
\medskip

We shall give  new short proofs of (\ref{eqfm}) and (\ref{trdet1}) next. 

\begin{proofof3}  
Since $H$ is positive semidefinite, so is $\widetilde{H}$, 
then the diagnal block matrices $G_{ll}$ are also positive semidifinite. 
By Fan-Ky's inequality \cite[p. 488]{Horn}, we have 
\[ \det \left(\sum\limits_{l=1}^k G_{ll} \right) \ge 
k^n\sqrt[k]{\prod\limits_{l=1}^k \det G_{ll}}. \]
By Lemma \ref{tr1tr2} and Fischer's inequality, we obtain
\[ \left( \frac{\det (\mathrm{tr}_2H)}{k^n}\right)^k
=\left( \frac{\det (\mathrm{tr}_1\widetilde{H})}{k^n}\right)^k 
\ge \prod\limits_{l=1}^k \det G_{ll}\ge \det \widetilde{H}=\det H.  \]
We get the result.
\end{proofof3}

\begin{proofof5} 
As  the diagnal block matrices $G_{ll}$ are positive semidefinite, 
by AM-GM inequality, we get 
\[ \frac{1}{k}\sum\limits_{l=1}^k\det G_{ll} \ge \sqrt[k]{\prod\limits_{l=1}^k\det G_{ll}}. \]
Combining Lemma \ref{tr1tr2} and (\ref{gll}), it yields 
\[ \left( \frac{\mathrm{tr}(\mathrm{det}_1 H)}{k}\right)^k
=\left( \frac{\mathrm{tr}(\mathrm{det}_2 \widetilde{H})}{k}\right)^k
\ge \prod\limits_{l=1}^k \det G_{ll}\ge \det \widetilde{H}=\det H.  \]
\end{proofof5}

In the above proofs, we actually 
use the symmetry of definitions of $\mathrm{tr}_1$ and $\mathrm{tr}_2$, 
$\mathrm{det}_1$ and $\mathrm{det}_2$. 
 As the byproducts of our argument, we have the following propositions 
by a trivial analysis. We omit the details  here.

\begin{proposition}
Let  $H\in \mathbb{M}_n(\mathbb{M}_k)$ be positive semidefinite. 
The following two inequalities are equivalent. 
\begin{align} \left( \frac{\det (\mathrm{tr}_1H)}{n^k}\right)^n \ge \det H,\\
\left( \frac{\det (\mathrm{tr}_2H)}{k^n}\right)^k \ge \det H. 
\end{align}
\end{proposition}

\begin{proposition}
Let  $H\in \mathbb{M}_n(\mathbb{M}_k)$ be positive semidefinite. 
The following two inequalities are equivalent. 
\begin{align}
\left( \frac{\mathrm{tr}(\mathrm{det}_1 H)}{k}\right)^k\ge \det H,\\
\left( \frac{\mathrm{tr}(\mathrm{det}_2 H)}{n}\right)^n\ge \det H. 
\end{align}
\end{proposition}

\section{Partial determinant inequalities}

If $A$ is positive semidefinite, then we write $A\ge 0$, and 
for two Hermitian matrices $A,B\in \mathbb{M}_n$, 
the symbol $A\ge B$ means that $A-B\ge 0$. 
In \cite{Lin16}, it is shown that if $A,B\in \mathbb{M}_n(\mathbb{M}_k)$ are positive 
semidefinite, then 
\begin{equation}\label{det2} 
\mathrm{det}_2 (A+B) \ge \mathrm{det}_2 A +\mathrm{det}_2 B.
\end{equation}
Choi \cite[Corollary 9]{Choi17} gave the corresponding complement as

\begin{equation}\label{det1}
 \mathrm{det}_1 (A+B) \ge \mathrm{det}_1 A +\mathrm{det}_1 B. 
\end{equation}
In what follows, we will   extend  (\ref{det2}) and (\ref{det1}) to a more generalized setting. 

\begin{lemma} \label{lem10}
Let $A=[A_{ij}]_{i,j=1}^n\in \mathbb{M}_n(\mathbb{M}_k)$. 
Then $[\otimes^r A_{ij}]_{i,j=1}^n $ is a principal submatrix of $\otimes^r A$.
\end{lemma}

\begin{proof}
Without loss of generality, we may write $A=X^*Y$, where $X,Y$ are $nk\times nk$. 
Now we partition $X=(X_1,X_2,\ldots ,X_n)$ and $Y=(Y_1,Y_2,\ldots ,Y_n)$ 
with each $X_i,Y_i$ is an $nk\times k$ complex matrix. 
Under this partition, we see that $A_{ij}=X_i^*Y_j$. 
Also we have $Y_j=YE_j$, where $E_j$ is a suitable $nk\times k$ matrix 
such that its $j$-th block  is extractly $I_k$ and otherwise $0$.  So we obtain
\[ \otimes^r A_{ij}=\otimes^r (X_i^*Y_j )= \otimes^r (E_i^*X^*YE_j)
=(\otimes^r E_i)^* (\otimes^r (X^*Y)) (\otimes^r E_j).\]
In other words, 
\[ [\otimes^r A_{ij}]_{i.j=1}^n =E^*(\otimes^rA)E, 
\quad E=[\otimes^r E_1,\otimes^r E_2,\ldots ,\otimes^r E_n]. \]
It is easy to verify that $E$ is a permutation matrix with $1$ only in diagonal entries.
\end{proof}

\begin{lemma} 
(\cite[Theorem 2.1]{BS15}) 
\label{lem32}
Let $A,B,C$ be positive semidefinite matrices of same size. 
Then for every positive integer $r$, we have
\begin{equation} \label{eq13}
\begin{aligned}
& \otimes^r (A+B+C) +\otimes^r A +\otimes^r B +\otimes^r C \\ 
&\quad \ge \otimes^r (A+B) +\otimes^r (A+C) +\otimes^r (B+C).
\end{aligned} \end{equation}
\end{lemma} 

\begin{proof}
For completeness, we include a proof by induction on $r$. 
The trivial case $r=1$  holds with equality, 
and the case $r=2$ is easy to verify. 
Assume therefore (\ref{eq13}) holds for some $r=m\ge 2$, that is 
\begin{equation*}
\begin{aligned}
& \otimes^m (A+B+C) +\otimes^m A +\otimes^m B +\otimes^m C \\ 
&\quad \ge \otimes^m (A+B) +\otimes^m (A+C) +\otimes^m (B+C).
\end{aligned} \end{equation*}
For $r=m+1$, we have  
\begin{align*}
&\otimes^{m+1}(A+B+C)\\
&\quad = \bigl( \otimes^m(A+B+C) \bigr) \otimes (A+B+C)\\
&\quad \ge \bigl( \otimes^m (A+B) +\otimes^m (A+C) +\otimes^m (B+C) -
\otimes^m A - \otimes^m B - \otimes^m C \bigr) \\
  &\quad\quad\,    \otimes (A+B+C) \\
&\quad = \otimes^{m+1}(A+B) +\otimes^{m+1}(A+C) +\otimes^{m+1}(B+C) \\
&\quad \quad     -\otimes^{m+1}A-\otimes^{m+1}B-\otimes^{m+1}C \\
&\quad \quad + \bigl( \otimes^m(A+B) \bigr) \otimes C + 
\bigl(\otimes^m(A+C)\bigr) \otimes B  + \bigl(\otimes^m(B+C)\bigr) \otimes A \\
&\quad \quad - \bigl(\otimes^m A \bigr) \otimes (B+C) -\bigl(\otimes^m B \bigr) \otimes (A+C) 
- \bigl(\otimes^m C \bigr) \otimes (A+B).
\end{align*}
It remains to show that 
\begin{align*}
&\bigl( \otimes^m(A+B) \bigr) \otimes C + 
\bigl(\otimes^m(A+C)\bigr) \otimes B  + \bigl(\otimes^m(B+C)\bigr) \otimes A \\
&\quad \ge  \bigl(\otimes^m A \bigr) \otimes (B+C)  + \bigl(\otimes^m B \bigr) \otimes (A+C) 
  + \bigl(\otimes^m C \bigr) \otimes (A+B).
\end{align*}
This follows immediately by the superadditivity of 
tensor power, by Proposition \ref{prop}, 
\begin{align*} 
\otimes^m (A+B)&\ge \otimes^m A +\otimes^m B, \\
\otimes^m(A+C) &\ge \otimes^m A +\otimes^m C, \\
\otimes^m(B+C) &\ge \otimes^m B +\otimes^m C. 
\end{align*}
Thus, the desired inequality (\ref{eq13}) holds.
\end{proof}

Tie et al. \cite[Lemma 2.2]{Tie11} established the following tensor product inequality (\ref{eq14}), 
we here demonstrate that it might be actually viewed as a corollary of Lemma \ref{lem32}. 

\begin{corollary} \label{coro33}
Let $A,B,C$ be positive semidefinite matrices of same size. 
Then for each positive integer $r$, we have
\begin{equation}\label{eq14}
\otimes^r (A+B+C) +\otimes^r C \ge \otimes^r (A+C) +\otimes^r (B+C).
\end{equation}
\end{corollary}

\begin{proof}
By Lemma \ref{lem32} and Proposition \ref{prop}, we obtain 
\begin{align*}
&\otimes^r (A+B+C)  +\otimes^r  C 
 - (\otimes^r (A+C) +  \otimes^r (B+C) ) \\
& \quad \ge \otimes^r (A+B)-\otimes^r A - \otimes^r B \ge 0. 
\end{align*}
The desired inequality (\ref{eq14}) follows.
\end{proof}

The next result Theorem \ref{thm14}  is an extension of (\ref{det2}) and (\ref{det1}). 

\begin{theorem} \label{thm14}
Let $A,B,C\in \mathbb{M}_n(\mathbb{M}_k)$ be positive semidefinite. Then
\begin{equation} \label{eq17} \begin{aligned} 
 &\mathrm{det}_1(A+B+C) +\mathrm{det}_1A +\mathrm{det}_1B +\mathrm{det}_1 C \\
 & \quad \ge \mathrm{det}_1(A+B) + \mathrm{det}_1 (A+C)+\mathrm{det}_1(B+C), 
\end{aligned}  \end{equation}
and 
\begin{equation} \label{eq18} \begin{aligned} 
& \mathrm{det}_2(A+B+C) +\mathrm{det}_2A +\mathrm{det}_2B +\mathrm{det}_2 C \\
& \quad  \ge \mathrm{det}_2(A+B) + \mathrm{det}_2 (A+C)+\mathrm{det}_2(B+C).
\end{aligned}  \end{equation}
\end{theorem}

\begin{proof}
We only prove (\ref{eq18}), 
and (\ref{eq17}) can be proved by exchanging the role of $\widetilde{A}$ and $A$. 
By Lemma \ref{lem32}, we have 
\begin{equation*} 
\begin{aligned}
& \otimes^r (A+B+C) +\otimes^r A +\otimes^r B +\otimes^r C \\ 
&\quad \ge \otimes^r (A+B) +\otimes^r (A+C) +\otimes^r (B+C).
\end{aligned} \end{equation*}
By Lemma \ref{lem10}, it yields 
\begin{align*} 
 & [\otimes^r (A_{ij}+B_{ij}+C_{ij})]_{i,j=1}^n 
+ [\otimes^r A_{ij}]_{i,j=1}^n + [\otimes^r B_{ij}]_{i,j=1}^n+[\otimes^r C_{ij}]_{i,j=1}^n \\
 &\quad \ge [\otimes^r (A_{ij}+B_{ij})]_{i,j=1}^n + 
[\otimes^r (A_{ij}+C_{ij})]_{i,j=1}^n + [\otimes^r (B_{ij}+C_{ij})]_{i,j=1}^n. 
\end{align*}
By restricting above inequality to the antisymmetric tensors, one obtains 
\begin{align*}
& [\wedge^r (A_{ij}+B_{ij}+C_{ij})]_{i,j=1}^n 
+ [\wedge^r A_{ij}]_{i,j=1}^n+ [\wedge^r B_{ij}]_{i,j=1}^n +[\wedge^r C_{ij}]_{i,j=1}^n \\
&\quad \ge [\wedge^r (A_{ij}+B_{ij})]_{i,j=1}^n +
[\wedge^r (A_{ij}+C_{ij})]_{i,j=1}^n + [\wedge^r (B_{ij}+C_{ij})]_{i,j=1}^n. 
\end{align*}
The required result (\ref{eq18}) follows by noting that $\det A_{ij}=\wedge^k A_{ij}$. 
\end{proof}

 \begin{corollary} \label{thm12}
Let $A,B,C\in \mathbb{M}_n(\mathbb{M}_k)$ be positive semidefinite. Then
\begin{equation}\label{eqdet1}
 \mathrm{det}_1(A+B+C) +\mathrm{det}_1C\ge \mathrm{det}_1(A+C) +\mathrm{det}_1(B+C),
\end{equation}
and 
\begin{equation}\label{eqdet2}
 \mathrm{det}_2(A+B+C) +\mathrm{det}_2C\ge \mathrm{det}_2(A+C) +\mathrm{det}_2(B+C). 
\end{equation}
\end{corollary}

\begin{proof}
Along the similar lines as in Theorem \ref{thm14}, 
it is not difficult to give the proof by applying Corollary \ref{coro33}. 
We leave the details for the reader. 
\end{proof}

\noindent
{\bf Remark}~~
It is worth noting that after finishing the first version of this paper, 
the referee informed the author that 
 (\ref{eqdet1}) and (\ref{eqdet2}) 
might be viewed as a corollary of Theorem \ref{thm14}. 
Since  by Theorem \ref{thm14} and (\ref{det1}), 
\begin{align*}
&\mathrm{det}_1(A+B+C)  +\mathrm{det}_1 C 
 - (\mathrm{det}_1 (A+C) + \mathrm{det}_1(B+C) ) \\
& \quad \ge \mathrm{det}_1(A+B)-\mathrm{det}_1A - \mathrm{det}_1B \ge 0. 
\end{align*}
Therefore, (\ref{eq17}) implies (\ref{eqdet1}). 
Similarly, (\ref{eq18}) implies (\ref{eqdet2}) by using (\ref{det2}). 

In particular, when $n=1$, (\ref{eqdet2}) is the  well-known determinantal inequality: 
\[ \det (A+B+C)+\det C \ge \det (A+C) +\det (B+C).  \]
And (\ref{eq18}) in Theorem \ref{thm14} reduces to the following result:  
\begin{align*} 
& \mathrm{det}(A+B+C) +\mathrm{det}A +\mathrm{det}B +\mathrm{det} C \\
&\quad \ge \mathrm{det}(A+B) + \mathrm{det} (A+C)+\mathrm{det}(B+C), 
\end{align*}
which is the main result obtained  in \cite{Lin14} 
by  using majorization theory.

\section*{Acknowledgments}
The work is supported by the Fundamental Research Funds for 
the Central Universities of Central South University. 
The first author would like to thank Dr. Minghua Lin, 
who introduced and encouraged him to the study of fascinating matrix theory
when he was an undergraduate at the HNNU. 
All authors are grateful for valuable comments from the referee, 
which considerably improve the presentation of our manuscript.

\end{document}